\documentclass{amsart}

\usepackage{a4wide}
\usepackage{cite}

\usepackage[T1]{fontenc}
\usepackage{textcomp}

\usepackage{amsmath}
\usepackage{amsfonts}
\usepackage{amsthm}
\usepackage{amssymb}
\usepackage{mathrsfs}

\newtheorem{theorem}{Theorem}[section]
\newtheorem{lemma}[theorem]{Lemma}
\newtheorem{prop}[theorem]{Proposition}

\theoremstyle{definition}
\newtheorem{definition}[theorem]{Definition}

\newtheorem{conjecture}[theorem]{Conjecture}
\newtheorem{remark}[theorem]{Remark}

\numberwithin{equation}{section}

\def\abs#1{\lvert#1\rvert}
\DeclareMathOperator*{\Res}{\mathrm{Res}}
\DeclareMathOperator*{\Ht}{\mathrm{ht}}
\DeclareMathOperator*{\Aut}{\mathrm{Aut}}
\DeclareMathOperator*{\Id}{\mathrm{id}}
\begin{document}

\title{Functional equations for Weng's zeta functions for $(G,P)/\mathbb{Q}$}
\author{Yasushi Komori}

\begin{abstract}
  It is shown that
  Weng's zeta functions associated with arbitrary semisimple algebraic
  groups defined over the rational number field and their maximal
  parabolic subgroups satisfy the functional equations.
\end{abstract}
\maketitle

\section{Introduction}
\label{sec:introduction}
Recently, Lin Weng introduced a new class of abelian zeta functions
associated to a pair of reductive algebraic group $G$ and its maximal
parabolic subgroup $P$, which are related with constant terms of
Eisenstein series.  In this paper, we simply refer to these zeta
functions as Weng's zeta functions.  These are motivated by and
closely related to non-abelian zeta functions called ``high rank zeta
functions'' associated with algebraic number fields, which were also
introduced by Weng himself from a viewpoint of Arakelov geometry based
on Iwasawa's interpretation and Tate's Fourier analysis on ad\'eles.
High rank zeta functions are generalizations of the Dedekind zeta
functions and in fact, rank one zeta functions coincide with the Dedekind
zeta functions up to constant multiples.  Hence the study of Weng's
zeta functions is not only interesting itself but also suggestive for
the study of the Dedekind zeta functions.  The profound background,
the path to the discovery, and the development of Weng's zeta
functions are detailed in his elaborated papers
\cite{Weng06a,Weng06b,Weng07,Weng10}.

One of the most significant properties for Weng's zeta functions is
the behavior of their zeros.  Weng conjectured that for any pair
$(G,P)$, Weng's zeta functions satisfy certain functional equations
and the Riemann hypothesis, as is expected or shown for various kinds
of zeta functions.  In fact,
in some special cases, it was shown in
\cite{LS06,Suz07,Suz08,SW08,Ki10} that they satisfy standard
functional equations and the Riemann hypothesis.

In this paper, we establish the functional equations in arbitrary
semisimple cases in a unified way.  We will see that the functional
equations are governed by the involutions on the Weyl groups (see the
last paragraph of Section \ref{sec:example} and Lemma
\ref{lm:sym_fg}).

Since the proofs known so far for the Riemann hypothesis for Weng's zeta
functions essentially use the functional equations, our result will be
a first and important step toward a comprehensive proof of the general
Riemann hypothesis.  Furthermore we give the explicit forms of Weng's
zeta functions and the precise description of the centers for the
functional equations (see \eqref{eq:center}), by which these zeta
functions will become more accessible than before.

This paper is organized as follows. In Section
\ref{sec:wengs-zeta-functions}, we give basic facts about root systems
and state the main results.  In Section \ref{sec:example}, to explain
the idea of the general proof, we demonstrate the proof of the
functional equation in a simple example, which also explains the
symbols used in the next sections.  In Section
\ref{sec:preliminaries}, we show some statements about properties of
Weyl groups and subsets of roots.  The last section is devoted to the
proof of the general functional equations.

\bigskip

Acknowledgement: The author would like to thank Lin Weng, Masatoshi
Suzuki and Hiroyuki Ochiai for fruitful discussion and critical
reading of the manuscript.  Thanks are also due to Kohji Matsumoto and
Hirofumi Tsumura for valuable comments.

\section{Weng's zeta functions and their functional equations}
\label{sec:wengs-zeta-functions}

We first fix notation and summarize basic facts about root systems and
Weyl groups. See \cite{Hum,Hum72,Bourbaki} for the details.  Let $V$
be an $r$-dimensional real vector space equipped with an inner product
$\langle \cdot,\cdot\rangle$.  Let $\Phi\subset V$ be a root system
of rank $r$ and $\Delta=\{\alpha_1,\ldots,\alpha_r\}$, its fundamental
system.  Let $\alpha^\vee=2\alpha/\langle\alpha,\alpha\rangle$
be the coroot associated with $\alpha\in\Phi$.  Let
$\Lambda=\{\lambda_1,\ldots,\lambda_r\}$ be the fundamental weights
satisfying $\langle \alpha_i^\vee,\lambda_j\rangle=\delta_{ij}$.  Let
$\Phi_+$ be the corresponding positive system of $\Phi$ and
$\Phi_-=-\Phi_+$ so that $\Phi=\Phi_+\cup\Phi_-$.  Let
\begin{equation}
\rho=\frac{1}{2}\sum_{\alpha\in\Phi_+}\alpha=\sum_{j=1}^r \lambda_j
\end{equation}
be the Weyl vector.  Let
$\Ht\alpha^\vee=\langle\rho,\alpha^\vee\rangle$ be the height of
$\alpha^\vee$.

Let $W$ be the Weyl group generated by simple reflections
$\sigma_j:V\to V$ for $\alpha_j$.  For $w\in W$, let
$l(w)=\abs{\Phi_w}$ be the length of $w$, where
$\Phi_w=\Phi_+\cap w^{-1}\Phi_-$.  Let $w_0$ be the longest
element of $W$.  Then we have $w_0^2=\Id$, $w_0\Delta=-\Delta$ and
$w_0\Phi_+=\Phi_-$.

Let $\Aut(\Phi)$ be the group of automorphisms of $V$ which
preserves $\Phi$. Then $W\subset\Aut(\Phi)$ and $W$ is a normal
subgroup of $\Aut(\Phi)$.  Let $\Gamma$ be the Dynkin diagram of
$\Phi$ and $\Aut(\Gamma)$, the group of automorphisms of $\Gamma$.
We identify $\Aut(\Gamma)$ with a group of permutations of indices
$\{1,\ldots,r\}$.  We also regard $\Aut(\Gamma)\subset\Aut(\Phi)$ in
a natural way.  For $\varpi\in\Aut(\Gamma)$, we have $\varpi\Delta=\Delta$
and $\varpi\Phi_+=\Phi_+$.  In fact, by use of the simple
transitivity of $W$ on positive systems, it is easily shown that
$\Aut(\Phi)=\Aut(\Gamma)\ltimes W$.  Since $-w_0\Delta=\Delta$, we have
$-\Id=\varpi_0 w_0$ for some $\varpi_0\in\Aut(\Gamma)$.  We see that
$\varpi_0^2=\Id$.

In the following, we fix $p$ with $1\leq p\leq r$.  Let $\Phi_p$ be
the root system normal to $\lambda_p$.
A fundamental
system of $\Phi_p$ is given by $\Delta_p=\Delta\setminus\{\alpha_p\}$.
Let $\Phi_{p+}=\Phi_p\cap\Phi_+\subset\Phi_+$ be the
corresponding positive system of $\Phi_p$.  Let
\begin{equation}
\rho_p=\frac{1}{2}\sum_{\alpha\in\Phi_{p+}}\alpha.    
\end{equation}
Note that $\rho_p\neq\sum_{j\neq p}^r\lambda_j$ in general.
Let $W_p$ be the Weyl group of $\Phi_p$.
Let $w_p$ be the longest
element of $W_p$.  Then we have $w_p^2=\Id$, $w_p\Delta_p=-\Delta_p$ and
$w_p\Phi_{p+}=\Phi_{p-}$.

Let $\mathbb{N}$ be the set of all positive integers.
Throughout this paper, we use the constants
\begin{equation}
\label{eq:center}
  c_p=2\langle\lambda_p-\rho_p,\alpha_p^\vee\rangle\in\mathbb{N},
\end{equation} 
which are important quantities describing the critical lines of Weng's
zeta functions.  Note that
\begin{equation}
  c_p=c_q
\end{equation}
for $q\in\Aut(\Gamma)p$.

Following \cite{Weng07,Weng10}, we introduce Weng's zeta function
associated with a semisimple algebraic group $G$ of rank $r$ 
defined over
the rational number field $\mathbb{Q}$ and its maximal parabolic
subgroup $P$.  
Let $\Phi$ be the root system of $G$, and $p$ be the index for which
a simple root $\alpha_p\in\Delta$ corresponds to $P$.  Similarly we use
the index $q$ corresponding to another maximal parabolic subgroup $Q$.
For the details of Weng's zeta functions, 
see \cite{Weng07,Weng10} and the references therein.

Let $\xi(s)=\pi^{-s/2}\Gamma(s/2)\zeta(s)$, where $\zeta$ is the
Riemann zeta function.  The poles of $\xi(s)$ are simple and on
$s=0,1$ with their residues being $-1,1$ respectively.  Moreover we
have the functional equation $\xi(1-s)=\xi(s)$. Then the period
$\omega^G_{\mathbb{Q}}(\lambda;T)$ for $G$ over $\mathbb{Q}$ is
defined as follows.
\begin{definition}[{Periods \cite[p.12, Fact E$'$]{Weng10}}]
  For $\lambda,T\in V$,
  \begin{align}
    \label{eq:def_period}
    \omega^G_{\mathbb{Q}}(\lambda;T)&=\sum_{w\in W}
    e^{\langle w\lambda-\rho,T\rangle}
    \biggl(
    \prod_{\alpha\in\Delta}
    \frac{1}{\langle w\lambda-\rho,\alpha^\vee\rangle}
    \biggr)
    \biggl(
    \prod_{\alpha\in\Phi_w}
    \frac{\xi(\langle\lambda,\alpha^\vee\rangle)}{\xi(\langle\lambda,\alpha^\vee\rangle+1)}
    \biggr),\\
    \omega^G_{\mathbb{Q}}(\lambda)&=\omega^G_{\mathbb{Q}}(\lambda;0).
  \end{align}
\end{definition}
Put $\Delta_p=\Delta\setminus\{\alpha_p\}=\{\beta_1,\ldots,\beta_{r-1}\}$ and $s=\langle\lambda-\rho,\alpha_p^\vee\rangle$. Let
\begin{equation}
\label{eq:def_omega}
  \omega^{G/P}_{\mathbb{Q}}(s;T)=
  \Res_{\langle\lambda-\rho,\beta_1^\vee\rangle=0}
  \cdots
  \Res_{\langle\lambda-\rho,\beta_{r-1}^\vee\rangle=0}
  \omega^G_{\mathbb{Q}}(\lambda;T).
\end{equation}
Then we have the explicit form of $\omega^{G/P}_{\mathbb{Q}}(s;T)$.
\begin{prop}
  \label{prop:exp_omega}
  $\omega^{G/P}_{\mathbb{Q}}(s;T)$ is independent of the ordering of $\Delta_p$
  and is
  given by
  \begin{equation}
    \label{eq:exp_period}
    \begin{split}
      \omega^{G/P}_{\mathbb{Q}}(s;T)
      &=
      \sum_{\substack{w\in W\\\Delta_p\subset w^{-1}(\Delta\cup \Phi_-)}}
      e^{\langle w(\rho+s\lambda_p)-\rho,T\rangle}
      \biggl(
      \prod_{\alpha\in(w^{-1}\Delta)\setminus\Delta_p}
      \frac{1}{\langle \lambda_p,\alpha^\vee\rangle s+\Ht\alpha^\vee-1}
      \biggr)
      \\
      &\qquad\qquad\times
      \biggl(
      \prod_{\alpha\in\Phi_w\setminus\Delta_p}
      \xi(\langle \lambda_p,\alpha^\vee\rangle s+\Ht\alpha^\vee)
      \biggr)
      \biggl(
      \prod_{\alpha\in (-\Phi_w)}
      \frac{1}
      {\xi(\langle \lambda_p,\alpha^\vee\rangle s+\Ht\alpha^\vee)}
      \biggr),
    \end{split}
  \end{equation}
\end{prop}
From this proposition,
we see that $\omega^{G/P}_{\mathbb{Q}}(s;T)$ is a sum of rational functions of $\xi$
functions.
Weng's zeta
function $\xi^{G/P}_{\mathbb{Q};o}(s;T)$ is defined
by multiplying the minimal numbers of $\xi$ functions such
that all the denominators are cancelled.
To describe the minimal $\xi$ factor, we need the following:
for $(k,h)\in\mathbb{Z}^2$,
\begin{multline}
\label{eq:def_M_p}
  M_p(k,h)=\max_{\substack{w\in W\\ \Delta_p\subset w^{-1}(\Delta\cup \Phi_-)}}
\bigl(\sharp\{\alpha\in w^{-1}\Phi_-~|~\langle\lambda_p,\alpha^\vee\rangle=k,\Ht\alpha^\vee=h-1\}
\\
-
\sharp\{\alpha\in w^{-1}\Phi_-~|~\langle\lambda_p,\alpha^\vee\rangle=k,\Ht\alpha^\vee=h\}\bigr).
\end{multline}
\begin{theorem}
  \label{thm:exp_xi}
  \begin{equation}
    \label{eq:exp_xi}
    \xi^{G/P}_{\mathbb{Q};o}(s;T)
    =    \omega^{G/P}_{\mathbb{Q}}(s;T)
    \prod_{k=0}^\infty
    \prod_{h=2}^\infty
    \xi(ks+h)^{M_p(k,h)}.
  \end{equation}
\end{theorem}
Note that $M_p(k,h)\neq 0$ for only finitely many pairs $(k,h)$
and the infinite products in this theorem should be understood as finite products.

Now we have the following functional equations
for $\xi^{G/P}_{\mathbb{Q};o}(s;T)$.
\begin{theorem}[Functional Equations]
  \label{thm:main1}
  \begin{equation}
    \begin{split}
      \xi^{G/P}_{\mathbb{Q};o}(-c_p-s;\varpi_0T)&=\xi^{G/P}_{\mathbb{Q};o}(s;T)
      \\
      &=\xi^{G/Q}_{\mathbb{Q};o}(s;\varpi T),
    \end{split}
  \end{equation}
  where 
  $\varpi\in\Aut(\Gamma)$ with
  $q=\varpi p$.
\end{theorem}
From the view point of the classical symmetry $s\leftrightarrow 1-s$,
we arrive at the following normalization and functional equations,
which immediately follow from Theorem \ref{thm:main1}.
\begin{definition}[Normalized Weng's zeta function]
  \begin{equation}
    \xi^{G/P}_{\mathbb{Q}}(s)=
    \xi^{G/P}_{\mathbb{Q};o}(s-(c_p+1)/2;0).
  \end{equation}
\end{definition}
\begin{theorem}[Functional Equations]
  \label{thm:main2}
  \begin{equation}
    \begin{split}
      \xi^{G/P}_{\mathbb{Q}}(1-s)&=\xi^{G/P}_{\mathbb{Q}}(s)\\
      &=\xi^{G/Q}_{\mathbb{Q}}(s),
    \end{split}
  \end{equation}
  where $q\in\Aut(\Gamma)p$.
\end{theorem}
\begin{conjecture}[Riemann Hypothesis \cite{Weng07,Weng10}]
  \label{conj:RH} 
  All zeros of the zeta function
  $\xi^{G/P}_{\mathbb{Q}}(s)$ lie on the central line $\Re s=\dfrac{1}{2}$.
\end{conjecture}
In the cases $A_1$, $A_2$, $B_2$ and $G_2$,
this conjecture was already confirmed
in \cite{LS06,Suz07,Suz08,SW08}. 

\begin{remark}
  In \cite{KHS10}, a weak version of Conjecture \ref{conj:RH} is proved in arbitrary root systems.
  Furthermore
  in \cite[Corollary 8.7]{KHS10}, a case-by-case investigation shows that
  the maximum in the definition 
  \eqref{eq:def_M_p}
  is attained by the longest element $w_0$, and hence we have
  \begin{equation}
    M_p(k,h)
    =
    \sharp\{\alpha\in \Phi_+~|~\langle\lambda_p,\alpha^\vee\rangle=k,\Ht\alpha^\vee=h-1\}
    -
    \sharp\{\alpha\in\Phi_+~|~\langle\lambda_p,\alpha^\vee\rangle=k,\Ht\alpha^\vee=h\}.
  \end{equation}
  In particular,
  \begin{equation}
    M_p(0,h)
    =
    \sharp\{\alpha\in \Phi_{p+}~|~\Ht\alpha^\vee=h-1\}
    -
    \sharp\{\alpha\in\Phi_{p+}~|~\Ht\alpha^\vee=h\}.
  \end{equation}
  Thus we obtain
  \begin{equation}
    \xi^{G/P}_{\mathbb{Q};o}(s;T)
    =
    \omega^{G/P}_{\mathbb{Q}}(s;T)
    \prod_{j=1}^{r-1}\xi(d_j)
    \prod_{k=1}^\infty
    \prod_{h=2}^\infty
    \xi(ks+h)^{M_p(k,h)},
  \end{equation}
where $d_j$ ($1\leq j\leq r-1$) are the 
degrees of the Weyl group $W_p$
(see \cite{Hum} for the details).
\end{remark}

\section{Example}
\label{sec:example}
To explain the idea and
to clarify the roles of the symbols appearing in this paper,
we give an example in the case of type $A_2$ (i.e.~$G=\mathrm{SL}(3)$) and $p=1$.

Let
$\Delta=\{\alpha_1,\alpha_2\}$ be a fundamental system
and
$\Phi_+=\{\alpha_1,\alpha_2,\alpha_1+\alpha_2\}$, the corresponding positive system,
 and $\rho=\alpha_1+\alpha_2$.
Let $\{\lambda_1,\lambda_2\}$ be the fundamental weights.
The Weyl group is given by
\begin{equation}
W=\{\Id,\sigma_1,\sigma_2,\sigma_1\sigma_2,\sigma_2\sigma_1,\sigma_1\sigma_2\sigma_1=\sigma_2\sigma_1\sigma_2=w_0\},
\end{equation}
where $w_0$ is the longest element.
We have $\Phi_{1+}=\{\alpha_2\}$ and the longest element $w_1=\sigma_2$ of the Weyl group of $\Phi_1$.

Put $\lambda=\rho+s_1\lambda_1+s_2\lambda_2$.
\begin{center}
  \begin{tabular}{cccc}
    & $w^{-1}\Delta$ & $\Phi_w=\Phi_+\cap w^{-1}\Phi_-$ & $w_0ww_1$\\
    \hline
    $\Id$ & $\{\alpha_1,\alpha_2\}$ & $\emptyset$ & $\sigma_2\sigma_1$\\
    $\sigma_1$ & $\{-\alpha_1,\alpha_1+\alpha_2\}$ & $\{\alpha_1\}$ & $\sigma_1$\\
    $\sigma_2$ & $\{\alpha_1+\alpha_2,-\alpha_2\}$ & $\{\alpha_2\}$ & $\sigma_1\sigma_2\sigma_1$\\
    $\sigma_2\sigma_1$ & $\{\alpha_2,-\alpha_1-\alpha_2\}$ & $\{\alpha_1,\alpha_1+\alpha_2\}$ & $\Id$ \\
    $\sigma_1\sigma_2$ & $\{-\alpha_1-\alpha_2,\alpha_1\}$ & $\{\alpha_2,\alpha_1+\alpha_2\}$ & $\sigma_1\sigma_2$ \\
    $\sigma_1\sigma_2\sigma_1=w_0$ & $\{-\alpha_1,-\alpha_2\}$ & $\{\alpha_1,\alpha_2,\alpha_1+\alpha_2\}=\Phi_+$ & $\sigma_2$
    \\
    \hline
  \end{tabular}
\end{center}
From the above table, we obtain
\begin{multline}
  \omega^G_{\mathbb{Q}}(\lambda)=
\frac{1}{s_1s_2}
+
\frac{1}{(-s_1-2)(s_1+s_2+1)}\frac{\xi(s_1+1)}{\xi(s_1+2)}
+
\frac{1}{(s_1+s_2+1)(-s_2-2)}\frac{\xi(s_2+1)}{\xi(s_2+2)}
\\
+
\frac{1}{s_2(-s_1-s_2-3)}\frac{\xi(s_1+1)\xi(s_1+s_2+2)}{\xi(s_1+2)\xi(s_1+s_2+3)}
+
\frac{1}{(-s_1-s_2-3)s_1}\frac{\xi(s_2+1)\xi(s_1+s_2+2)}{\xi(s_2+2)\xi(s_1+s_2+3)}
\\
+
\frac{1}{(-s_2-2)(-s_1-2)}\frac{\xi(s_1+1)\xi(s_2+1)\xi(s_1+s_2+2)}{\xi(s_1+2)\xi(s_2+2)\xi(s_1+s_2+3)}.
\end{multline}
By putting $s_1=s$ and taking the residue at $s_2=0$, we obtain
\begin{multline}
\label{eq:omega_1}
  \omega^{G/P}_{\mathbb{Q}}(s)=\Res_{s_2=0}\omega^G_{\mathbb{Q}}(\lambda)
=
\frac{1}{s}
+
0
+
\frac{1}{(s+1)(-2)}\frac{1}{\xi(2)}
\\
+
\frac{1}{(-s-3)}\frac{\xi(s+1)\xi(s+2)}{\xi(s+2)\xi(s+3)}
+
\frac{1}{(-s-3)s}\frac{\xi(s+2)}{\xi(2)\xi(s+3)}
\\
+
\frac{1}{(-2)(-s-2)}\frac{\xi(s+1)\xi(s+2)}{\xi(s+2)\xi(2)\xi(s+3)}.
\end{multline}
By multiplying the formal common $\xi$ factor 
\begin{equation}
 F_1(s)=\xi(-s-1)\xi(-1)\xi(-s-2)=\xi(s+2)\xi(2)\xi(s+3),  
\end{equation}
we define
\begin{multline}
  Z_1(s)=F_1(s)\omega^{G/P}_{\mathbb{Q}}(s)
\\
=
\frac{1}{s}\xi(s+2)\xi(2)\xi(s+3)
+
0
+
\frac{1}{(s+1)(-2)}\xi(s+2)\xi(s+3)
\\
+
\frac{1}{(-s-3)}\xi{2}\xi(s+1)\xi(s+2)
+
\frac{1}{(-s-3)s}\xi(s+2)^2
\\
+
\frac{1}{(-2)(-s-2)}\xi(s+1)\xi(s+2).
\end{multline}
It can be directly checked that
\begin{equation}
\label{eq:sym_Z}
  Z_1(-3-s)=Z_1(s),
\end{equation}
where the term corresponding to $w$ is exchanged for that
corresponding to $w_0ww_1$.  Note that $2\rho_1=\alpha_2$ and
\begin{equation}
  c_1=2\langle\lambda_1-\rho_1,\alpha_1^\vee\rangle=3.
\end{equation}

We have shown that $Z_1(s)$ satisfies the functional equation.  It is,
however, not Weng's zeta function because $F_1(s)$ is not the minimal
$\xi$ factor.  To obtain Weng's zeta function, we need the minimal
$\xi$ factor such that all the true denominators are cancelled in
$\omega^{G/P}_{\mathbb{Q}}(s)$. It is read off from \eqref{eq:omega_1} as
\begin{equation}
\label{eq:fac_A1}
  \xi(2)\xi(s+3)=\frac{F_1(s)}{D_1(s)},
\end{equation}
where $D_1(s)=\xi(s+2)$, which itself satisfies the functional equation
\begin{equation}
\label{eq:sym_D}
  D_1(-3-s)=D_1(s).
\end{equation}

Due to the symmetries \eqref{eq:sym_Z} and \eqref{eq:sym_D},
we conclude that Weng's zeta function
\begin{equation}
\xi^{G/P}_{\mathbb{Q};o}(s)=\Bigl(\frac{F_1(s)}{D_1(s)}\Bigr)
\omega^{G/P}_{\mathbb{Q}}(s)
=\frac{Z_1(s)}{D_1(s)}
\end{equation}
satisfies the functional equation
\begin{equation}
  \xi^{G/P}_{\mathbb{Q};o}(-3-s)=\xi^{G/P}_{\mathbb{Q};o}(s).
\end{equation}
Note that in \eqref{eq:fac_A1}, we see that $\xi(2)=\xi(d_1)$, where $d_1=2$ is the degree of the Weyl group of type $A_1$.

In general cases, this procedure works well and we prove the
functional equations in the following sections in this strategy.
As we remarked in the introduction,
we see that the map $\iota:W\to W$ defined by $w\mapsto
w_0ww_p$ plays an important role
in \eqref{eq:sym_Z};
$\iota$ is an involution, namely $\iota^2=\Id$, and governs the
functional equations at the level of the Weyl group.

\section{Preliminaries}
\label{sec:preliminaries}
In this section, we prove some statements about root systems
which is used in the proof of the functional equations.
\begin{lemma}
\label{lm:rho}
  \begin{equation}
      c_p\lambda_p-w_p\rho=\rho.
  \end{equation}
\end{lemma}
\begin{proof}
For $\alpha\in\Phi_{p+}$, we have $w_p\alpha\in \Phi_{p-}\subset\Phi_-$ by the definition of $w_p$.
For $\alpha\in\Phi_+\setminus\Phi_{p+}$, we have $w_p\alpha\in\Phi_+$
since $\alpha$ is of the form $a_p\alpha_p+\cdots$ with $a_p>0$
and $w_p\alpha=a_p\alpha_p+\cdots$ remains positive. Hence
we obtain
\begin{equation}
  \Phi_{w_p}=\Phi_+\cap w_p^{-1}\Phi_-=\Phi_+\cap(-w_p\Phi_+)=\Phi_{p+}
\end{equation}
and
\begin{equation}
\label{eq:rho_1}
  w_p\rho=\rho-\sum_{\alpha\in\Phi_{w_p}}\alpha
=\rho-\sum_{\alpha\in\Phi_{p+}}\alpha
=\rho-2\rho_p.
\end{equation}
By the property
$\sigma_k\Phi_{p+}=(\Phi_{p+}\setminus\{\alpha_k\})\cup\{-\alpha_k\}$
for $k\neq p$, we have
 \begin{equation}
   \sigma_k\rho_p=\rho_p-\alpha_k=\rho_p-\langle\rho_p,\alpha_k^\vee\rangle\alpha_k,
 \end{equation}
which implies $\langle\rho_p,\alpha_k^\vee\rangle=1$.
Therefore
\begin{equation}
\label{eq:rho_2}
  \rho_p
=\sum_{k=1}^r\langle\rho_p,\alpha_k^\vee\rangle\lambda_k
=\sum_{k\neq p}\lambda_k+\langle\rho_p,\alpha_p^\vee\rangle\lambda_p
  =\rho+\langle\rho_p-\lambda_p,\alpha_p^\vee\rangle\lambda_p.
\end{equation}
Combining \eqref{eq:rho_1} and \eqref{eq:rho_2},
we have
\begin{equation}
  c_p\lambda_p-w_p\rho=\rho+(c_p+2\langle\rho_p-\lambda_p,\alpha_p^\vee\rangle)\lambda_p=\rho.
\end{equation}
\end{proof}

\begin{lemma}
\label{lm:bij_w}
\begin{enumerate}
\item 
  For $w\in W$,
  $\Delta_p\subset w^{-1}(\Delta\cup \Phi_-)$ if and only if
  $\Delta_p\subset w_pw^{-1}w_0(\Delta\cup \Phi_-)$.
\item
  For $w\in W$ and $\varpi\in\Aut(\Gamma)$ with $q=\varpi p$,
  $\Delta_p\subset w^{-1}(\Delta\cup \Phi_-)$ if and only if
  $\Delta_q\subset \varpi w^{-1}\varpi^{-1}(\Delta\cup \Phi_-)$.
\end{enumerate}
\end{lemma}  
\begin{proof}
  \begin{enumerate}
  \item 
  We see that $\Delta_p\subset w_pw^{-1}w_0(\Delta\cup \Phi_-)$ is equivalent to
  $-\Delta_p\subset w^{-1}(-\Delta\cup \Phi_+)$
  and hence to
  $\Delta_p\subset w^{-1}(\Delta\cup \Phi_-)$.
  \item It follows from $\varpi\Delta_p=\Delta_q$,
    $\varpi\Delta=\Delta$ and $\varpi\Phi_-=\Phi_-$.
  \end{enumerate}
\end{proof}

For $w\in W$ and $(k,h)\in\mathbb{Z}^2$, 
let 
\begin{equation}
  \begin{split}
    N_{p,w}(k,h)&=\sharp\{\alpha\in w^{-1}\Phi_-~|~\langle\lambda_p,\alpha^\vee\rangle=k,\Ht\alpha^\vee=h\},
    \\
    N_p(k,h)&=
    \sharp\{\alpha\in\Phi~|~\langle\lambda_p,\alpha^\vee\rangle=k,\Ht\alpha^\vee=h\}.
  \end{split}
\end{equation}
We note that $N_{p,w}(k,h)\neq 0$ for finite numbers of pairs $(k,h)\in\mathbb{Z}^2$
and that for $(k,h)\in\mathbb{Z}^2$ with $k\geq1$ or $h\geq1$,
\begin{equation}
  \begin{split}
N_{p,w}(k,h)&=\sharp\{\alpha\in \Phi_+\cap w^{-1}\Phi_-~|~\langle\lambda_p,\alpha^\vee\rangle=k,\Ht\alpha^\vee=h\},
\\
N_p(k,h)&=
\sharp\{\alpha\in\Phi_+~|~\langle\lambda_p,\alpha^\vee\rangle=k,\Ht\alpha^\vee=h\}
\end{split}
\end{equation}
because $\alpha\in\Phi$ is either $\alpha\in\Phi_+$ or $\alpha\in\Phi_-$.

Consider the character of the dual Lie algebra ignoring the Cartan
subalgebra
\begin{equation}
X(\nu)=\sum_{\alpha\in\Phi}e^{\alpha^\vee}(\nu)
\end{equation}
for $\nu\in V$, where
\begin{equation}
e^{\alpha^\vee}(\nu)=e^{\langle\nu,\alpha^\vee\rangle}
\end{equation}
as usual.  Then
\begin{equation}
\label{eq:def_X}
  X(t\lambda_p+\rho)
  =
  \sum_{\alpha\in\Phi}e^{\langle\lambda_p,\alpha^\vee\rangle t+\Ht\alpha^\vee}
  =
  \sum_{k=-\infty}^\infty\sum_{h=-\infty}^\infty
  N_p(k,h)e^{kt+h}.
\end{equation}
Note that for $\nu\in V$ and $w\in\Aut(\Phi)$,
\begin{equation}
  X(\nu)=X(w\nu).
\end{equation}
\begin{lemma}
  \label{lm:N}
  \begin{enumerate}
  \item 
    For $(k,h)\in\mathbb{Z}^2$,
    \begin{equation}
      \label{eq:NN1}
      N_p(k,kc_p-h)=N_p(k,h).
    \end{equation}
  \item
    For $(k,h)\in\mathbb{Z}^2$ and $q\in\Aut(\Gamma)p$,
    \begin{equation}
      \label{eq:NN2}
      N_p(k,h)=N_q(k,h).
    \end{equation}
  \end{enumerate}
\end{lemma}
\begin{proof}
  \begin{enumerate}
  \item We have
    \begin{equation}
      X((c_p+t)\lambda_p-\rho)
        =X(t\lambda_p+w_p\rho)
        =X(w_p(t\lambda_p+\rho))
        =X(t\lambda_p+\rho)
    \end{equation}
    by Lemma \ref{lm:rho}.  
    Hence \eqref{eq:NN1} by comparing the coefficients.
  \item
    Since for $\varpi\in\Aut(\Phi)$ such that $q=\varpi p$,
    \begin{equation}
      X(t\lambda_p+\rho)
      =X(\varpi(t\lambda_p+\rho))
      =X(t\lambda_q+\rho),
    \end{equation}
    we have \eqref{eq:NN2}.
  \end{enumerate}
\end{proof}

\begin{lemma}
  \label{lm:Nw}
  \begin{enumerate}
  \item 
    For $(k,h)\in\mathbb{Z}^2$, 
    \begin{equation}
      \label{eq:NNw1}
      N_p(k,h)
      -N_{p,w_0ww_p}(k,kc_p-h)=
      N_{p,w}(k,h).
    \end{equation}
  \item
    For $(k,h)\in\mathbb{Z}^2$ and
    $\varpi\in\Aut(\Gamma)$ with $q=\varpi p$,
    \begin{equation}
      \label{eq:NNw2}
      N_{p,w}(k,h)=
      N_{q,\varpi w\varpi^{-1}}(k,h).
    \end{equation}
  \end{enumerate}
\end{lemma}
\begin{proof}
  \begin{enumerate}
  \item Since
    \begin{equation}
      \begin{split}
        \Phi
        &=w^{-1}\Phi_-\cup
        w^{-1}\Phi_+
        \\
        &=w^{-1}\Phi_-\cup
        w_p(w_pw^{-1}w_0)\Phi_-,
      \end{split}
    \end{equation}
    we have
    \begin{equation}
      \begin{split}
        X(t\lambda_p+\rho)
        &=
        \sum_{\alpha\in w^{-1}\Phi_-}e^{\alpha^\vee}(t\lambda_p+\rho)
        +
        \sum_{\alpha\in w_p(w_0ww_p)^{-1}\Phi_-}e^{\alpha^\vee}(t\lambda_p+\rho)
        \\
        &=
        \sum_{\alpha\in w^{-1}\Phi_-}e^{\alpha^\vee}(t\lambda_p+\rho)
        +
        \sum_{\alpha\in (w_0ww_p)^{-1}\Phi_-}e^{\alpha^\vee}((c_p+t)\lambda_p-\rho).
      \end{split}
    \end{equation}
    By comparing this with \eqref{eq:def_X},
    we obtain \eqref{eq:NNw1}. 
  \item 
    We have
    \begin{equation}
      \sum_{\alpha\in w^{-1}\Phi_-}e^{\alpha^\vee}(t\lambda_p+\rho)
      =
      \sum_{\alpha\in \varpi w^{-1}\varpi^{-1}\Phi_-}e^{\alpha^\vee}(\varpi(t\lambda_p+\rho))
      =
      \sum_{\alpha\in (\varpi w\varpi^{-1})^{-1}\Phi_-}e^{\alpha^\vee}(t\lambda_q+\rho),
    \end{equation}
    which implies \eqref{eq:NNw2}.
  \end{enumerate}
\end{proof}

\section{Proof of the functional equations}
\label{sec:proof}
\begin{proof}[Proof of Proposition \ref{prop:exp_omega}]
Put the coordinate
\begin{equation}
  \lambda=\sum_{k=1}^r(1+s_k)\lambda_k
=\rho+\sum_{k=1}^r s_k \lambda_k,
\end{equation}
so that for $\alpha^\vee=\sum_{k=1}^r a_k\alpha_k^\vee$,
\begin{equation}
\langle\lambda-\rho,\alpha^\vee\rangle=\sum_{k=1}^r a_ks_k.
\end{equation}

For $w\in W$,
the corresponding term in \eqref{eq:def_period} besides the exponential factor is calculated as
\begin{multline}
  \label{eq:term_period}
  A_w=
  \biggl(
  \prod_{\alpha\in\Delta}
  \frac{1}{\langle w\lambda-\rho,\alpha^\vee\rangle}
  \biggr)
  \biggl(
  \prod_{\alpha\in\Phi_w}
  \frac{\xi(\langle\lambda,\alpha^\vee\rangle)}{\xi(\langle\lambda,\alpha^\vee\rangle+1)}
  \biggr)
  \\
  \begin{aligned}
  &=
  \biggl(
  \prod_{\alpha\in\Delta}
  \frac{1}{\langle w\lambda,\alpha^\vee\rangle-1}
  \biggr)
  \biggl(
  \prod_{\alpha\in\Phi_w\cap\Delta_p}
  \frac{1}{\langle\lambda,\alpha^\vee\rangle-1}
  \biggr)
\\
&\qquad\qquad\times
  \biggl(
  \prod_{\alpha\in\Phi_w\cap\Delta_p}
  \frac{(\langle \lambda,\alpha^\vee\rangle-1)\xi(\langle\lambda,\alpha^\vee\rangle)}{\xi(\langle\lambda,\alpha^\vee\rangle+1)}
  \biggr)
  \biggl(
  \prod_{\alpha\in\Phi_w\setminus\Delta_p}
  \frac{\xi(\langle\lambda,\alpha^\vee\rangle)}{\xi(\langle\lambda,\alpha^\vee\rangle+1)}
  \biggr)
\\
  &=
  \biggl(
  \prod_{\alpha\in (w^{-1}\Delta\cup \Phi_w)\cap\Delta_p}
  \frac{1}{\langle \lambda,\alpha^\vee\rangle-1}
  \biggr)
  \biggl(
  \prod_{\alpha\in (w^{-1}\Delta)\setminus\Delta_p}
  \frac{1}{\langle \lambda,\alpha^\vee\rangle-1}
  \biggr)
\\
&\qquad\qquad\times
  \biggl(
  \prod_{\alpha\in\Phi_w\cap\Delta_p}
  \frac{(\langle \lambda,\alpha^\vee\rangle-1)\xi(\langle\lambda,\alpha^\vee\rangle)}{\xi(\langle\lambda,\alpha^\vee\rangle+1)}
  \biggr)
  \biggl(
  \prod_{\alpha\in\Phi_w\setminus\Delta_p}
  \frac{\xi(\langle\lambda,\alpha^\vee\rangle)}{\xi(\langle\lambda,\alpha^\vee\rangle+1)}
  \biggr).
\end{aligned}
\end{multline}

In order to put $s_p=s$ and take all the residues at $s_k=0$ for $k\neq p$ in \eqref{eq:term_period},
first we consider the third factor of the last member of \eqref{eq:term_period}.
For $\alpha_k\in\Phi_w\cap\Delta_p$, we have
\begin{equation}\label{eq:term_period1}
  \frac{(\langle \lambda,\alpha_k^\vee\rangle-1)\xi(\langle\lambda,\alpha_k^\vee\rangle)}{\xi(\langle\lambda,\alpha_k^\vee\rangle+1)}
=\frac{s_k\xi(s_k+1)}{\xi(s_k+2)}=\frac{1}{\xi(2)}+o(s_k)
\end{equation}
when $s_k\to 0$.

In the last factor,
for
$\alpha\in\Phi_w\setminus\Delta_p$, we have
\begin{equation}
\label{eq:term_period2}
  \frac{\xi(\langle\lambda,\alpha^\vee\rangle)}{\xi(\langle\lambda,\alpha^\vee\rangle+1)}
=
  \frac{\xi(\langle\lambda_p,\alpha^\vee\rangle s+\Ht\alpha^\vee)}
{\xi(\langle\lambda_p,\alpha^\vee\rangle s+\Ht\alpha^\vee+1)}
\end{equation}
when $s_k=0$ for $k\neq p$ and $s_p=s$.
If $\langle\lambda_p,\alpha^\vee\rangle=0$, then
$\alpha\in\Phi_{p+}\setminus\Delta_p$ and hence $\Ht\alpha^\vee\geq2$.
Thus
we see that
\eqref{eq:term_period2} is finite
if $\langle\lambda_p,\alpha^\vee\rangle=0$,
due to $\Ht\alpha^\vee\geq 2$.
Moreover it is also finite
 for appropriate $s\in\mathbb{C}$
if $\langle\lambda_p,\alpha^\vee\rangle\neq 0$.

We consider the second factor of the last member of \eqref{eq:term_period}.
When $s_k=0$ for $k\neq p$ and $s_p=s$,
we have
\begin{equation}
\label{eq:term_period3}
\langle\lambda,\alpha^\vee\rangle-1
=
  \langle\lambda-\rho,\alpha^\vee\rangle+\Ht\alpha^\vee-1
=
  \langle\lambda_p,\alpha^\vee\rangle s+\Ht\alpha^\vee-1.
\end{equation}
Since
 for $\alpha\in (w^{-1}\Delta)\setminus\Delta_p$,
 $\Ht\alpha^\vee\neq 1$ or $\langle\lambda_p,\alpha^\vee\rangle\neq0$ holds,
 \eqref{eq:term_period3} does not vanish identically.

The first factor is calculated as
\begin{equation}
\label{eq:term_period4}
  \prod_{\alpha\in (w^{-1}\Delta\cup \Phi_w)\cap\Delta_p}
  \frac{1}{\langle \lambda,\alpha^\vee\rangle-1}
=  \prod_{\alpha_k\in (w^{-1}\Delta\cup \Phi_w)\cap\Delta_p}
  \frac{1}{s_k}.
\end{equation}
Hence from \eqref{eq:term_period1}, \eqref{eq:term_period2},
\eqref{eq:term_period3} and \eqref{eq:term_period4}, we see that when
we take all the residues, only the terms with $\Delta_p\subset
w^{-1}\Delta\cup \Phi_w$ survive and the others vanish. In the former cases,
we obtain
\begin{multline}
  \Res_{\substack{s_k=0\\k\neq p}}
  A_w=\biggl(
  \prod_{\alpha\in(w^{-1}\Delta)\setminus\Delta_p}
  \frac{1}{\langle \lambda_p,\alpha^\vee\rangle s+\Ht\alpha^\vee-1}
  \biggr)
  \\
  \times
  \biggl(
  \prod_{\alpha\in\Phi_w\cap\Delta_p}
  \frac{1}{\xi(2)}
  \biggr)
  \biggl(
  \prod_{\alpha\in\Phi_w\setminus\Delta_p}
  \frac{\xi(\langle \lambda_p,\alpha^\vee\rangle s+\Ht\alpha^\vee)}
  {\xi(\langle \lambda_p,\alpha^\vee\rangle s+\Ht\alpha^\vee+1)}
  \biggr),
\end{multline}
which does not depend on the ordering of $\Delta_p$.

 Note that 
$\Delta_p\subset w^{-1}\Delta\cup \Phi_w$ if and only if $\Delta_p\subset
w^{-1}(\Delta\cup \Phi_-)$. 
If we put $s_p=s$ and take all
the residues at $s_k=0$ for $k\neq p$ in \eqref{eq:def_period},
we get
\begin{equation}
  \begin{split}
    \omega^{G/P}_{\mathbb{Q}}(s;T)
    &=
    \sum_{\substack{w\in W\\\Delta_p\subset w^{-1}(\Delta\cup \Phi_-)}}
    e^{\langle w(\rho+s\lambda_p)-\rho,T\rangle}
    \biggl(
    \prod_{\alpha\in(w^{-1}\Delta)\setminus\Delta_p}
    \frac{1}{\langle \lambda_p,\alpha^\vee\rangle s+\Ht\alpha^\vee-1}
    \biggr)
    \\
    &\qquad\qquad\times
    \biggl(
    \prod_{\alpha\in\Phi_w\cap\Delta_p}
    \frac{1}{\xi(2)}
    \biggr)
    \biggl(
    \prod_{\alpha\in\Phi_w\setminus\Delta_p}
    \frac{\xi(\langle \lambda_p,\alpha^\vee\rangle s+\Ht\alpha^\vee)}
    {\xi(\langle \lambda_p,\alpha^\vee\rangle s+\Ht\alpha^\vee+1)}
    \biggr)
    \\
    &=
    \sum_{\substack{w\in W\\\Delta_p\subset w^{-1}(\Delta\cup \Phi_-)}}
    e^{\langle w(\rho+s\lambda_p)-\rho,T\rangle}
    \biggl(
    \prod_{\alpha\in(w^{-1}\Delta)\setminus\Delta_p}
    \frac{1}{\langle \lambda_p,\alpha^\vee\rangle s+\Ht\alpha^\vee-1}
    \biggr)
    \\
    &\qquad\qquad\times
    \biggl(
    \prod_{\alpha\in\Phi_w\setminus\Delta_p}
    \xi(\langle \lambda_p,\alpha^\vee\rangle s+\Ht\alpha^\vee)
    \biggr)
    \biggl(
    \prod_{\alpha\in\Phi_w}
    \frac{1}
    {\xi(\langle \lambda_p,\alpha^\vee\rangle s+\Ht\alpha^\vee+1)}
    \biggr)
    \\
    &=
    \sum_{\substack{w\in W\\\Delta_p\subset w^{-1}(\Delta\cup \Phi_-)}}
    e^{\langle w(\rho+s\lambda_p)-\rho,T\rangle}
    \biggl(
    \prod_{\alpha\in(w^{-1}\Delta)\setminus\Delta_p}
    \frac{1}{\langle \lambda_p,\alpha^\vee\rangle s+\Ht\alpha^\vee-1}
    \biggr)
    \\
    &\qquad\qquad\times
    \biggl(
    \prod_{\alpha\in\Phi_w\setminus\Delta_p}
    \xi(\langle \lambda_p,\alpha^\vee\rangle s+\Ht\alpha^\vee)
    \biggr)
    \biggl(
    \prod_{\alpha\in (-\Phi_w)}
    \frac{1}
    {\xi(\langle \lambda_p,\alpha^\vee\rangle s+\Ht\alpha^\vee)}
    \biggr),
  \end{split}
\end{equation}
where in the last equality, we used the functional equation for $\xi(s)$.
\end{proof}

Let
\begin{equation}
\label{eq:def_F_p}
  F_p(s)
  =
  \prod_{\alpha\in\Phi_-}
  \xi(\langle \lambda_p,\alpha^\vee\rangle s+\Ht\alpha^\vee),
\end{equation}
and define
\begin{equation}
\label{eq:def_Z_p}
    Z_p(s;T)=
    F_p(s)\omega^{G/P}_{\mathbb{Q}}(s;T).
\end{equation}
Then we have the following proposition.
\begin{prop}[Functional Equations]
\label{prop:main}
\begin{equation}
\label{eq:main}
  \begin{split}
Z_p(-c_p-s;\varpi_0 T)&=Z_p(s;T)
\\
&=Z_q(s;\varpi T),
\end{split}
\end{equation}
where 
$\varpi\in\Aut(\Gamma)$ with
$q=\varpi p$.
\end{prop}
To show this proposition, we need the explicit form of $Z_p(s;T)$.
\begin{prop}
\label{prop:explicit}
  \begin{equation}
    \label{eq:exp_Z_p}
    \begin{split}
      Z_p(s;T)
      &=
      \sum_{\substack{w\in W\\\Delta_p\subset w^{-1}(\Delta\cup \Phi_-)}}
      e^{\langle w(\rho+s\lambda_p)-\rho,T\rangle}
      \biggl(
      \prod_{\alpha\in(w^{-1}\Delta)\setminus\Delta_p}
      \frac{1}{\langle \lambda_p,\alpha^\vee\rangle s+\Ht\alpha^\vee-1}
      \biggr)
      \\
      &\qquad\qquad\times
      \biggl(
      \prod_{\alpha\in(w^{-1}\Phi_-)\setminus\Delta_p}
      \xi(\langle \lambda_p,\alpha^\vee\rangle s+\Ht\alpha^\vee)
      \biggr).
    \end{split}
  \end{equation}
\end{prop}
\begin{proof}
Since
\begin{equation}
  \begin{split}
\Phi_-\setminus(-\Phi_w)
&=\Phi_-\setminus(\Phi_-\cap w^{-1}\Phi_+)
\\
&=\Phi_-\setminus w^{-1}\Phi_+
\\
&=\Phi_-\cap w^{-1}\Phi_-,
\end{split}
\end{equation}
we have
\begin{multline}
    Z_p(s;T)=
    \sum_{\substack{w\in W\\\Delta_p\subset w^{-1}(\Delta\cup \Phi_-)}}
    e^{\langle w(\rho+s\lambda_p)-\rho,T\rangle}
    \biggl(
    \prod_{\alpha\in(w^{-1}\Delta)\setminus\Delta_p}
    \frac{1}{\langle \lambda_p,\alpha^\vee\rangle s+\Ht\alpha^\vee-1}
    \biggr)
    \\
\times
    \biggl(
    \prod_{\alpha\in(\Phi_w\setminus\Delta_p)\cup(\Phi_-\cap w^{-1}\Phi_-)}
    \xi(\langle \lambda_p,\alpha^\vee\rangle s+\Ht\alpha^\vee)
    \biggr).
  \end{multline}
  Using
\begin{equation}
  \begin{split}
(\Phi_w\setminus\Delta_p)\cup
(\Phi_-\cap w^{-1}\Phi_-)
&=  
((\Phi_+\cap w^{-1}\Phi_-)\setminus\Delta_p)
\cup(\Phi_-\cap w^{-1}\Phi_-)
\\
&=  
((\Phi_+\cap w^{-1}\Phi_-)
\cup(\Phi_-\cap w^{-1}\Phi_-)
)\setminus\Delta_p
\\
&=  
w^{-1}\Phi_-\setminus\Delta_p,
\end{split}
\end{equation}
we arrive at \eqref{eq:exp_Z_p}.
\end{proof}

Let
\begin{align}
  f_{p,w}(s)
  &=
  \prod_{\alpha\in(w^{-1}\Delta)\setminus\Delta_p}
  \frac{1}{\langle \lambda_p,\alpha^\vee\rangle s+\Ht\alpha^\vee-1},
  \\
  g_{p,w}(s)
  &=
  \prod_{\alpha\in(w^{-1}\Phi_-)\setminus\Delta_p}
  \xi(\langle \lambda_p,\alpha^\vee\rangle s+\Ht\alpha^\vee),
\end{align}
so that 
\begin{equation}
  Z_p(s;T)=
      \sum_{\substack{w\in W\\\Delta_p\subset w^{-1}(\Delta\cup \Phi_-)}}
      e^{\langle w(\rho+s\lambda_p)-\rho,T\rangle}f_{p,w}(s)g_{p,w}(s).
\end{equation}
\begin{lemma}
  \label{lm:sym_fg}
For $w\in W$ and
$\varpi\in\Aut(\Gamma)$ with
$q=\varpi p$,
  \begin{alignat}{2}
\label{eq:sym_fg1}
    f_{p,w}(-c_p-s)&=f_{p,w_0ww_p}(s), &\qquad
    g_{p,w}(-c_p-s)&=g_{p,w_0ww_p}(s), \\
\label{eq:sym_fg2}
    f_{p,\varpi^{-1}w\varpi}(s)&=f_{q,w}(s), &\qquad
    g_{p,\varpi^{-1}w\varpi}(s)&=g_{q,w}(s).
  \end{alignat}
\end{lemma}
\begin{proof}
  Fix $w\in W$. Then
  for a subset $A\subset\Phi$ with $A=\Delta$ or $\Phi_-$, we have
  $w_0A=-A$ and
  \begin{equation}
    \label{eq:f1}
    \begin{split}
      -w_p(w^{-1}A\setminus\Delta_p)
      &=
      (w_pw^{-1}(-A))\setminus(w_p(-\Delta_p))
      \\
      &=
      (w_pw^{-1}w_0A)\setminus \Delta_p.
    \end{split}
  \end{equation}
  Let
  \begin{equation}
    S(A;s;p,w)=\{\langle \lambda_p,\alpha^\vee\rangle s+\Ht\alpha^\vee~|~\alpha\in(w^{-1}A)\setminus\Delta_p\}
  \end{equation}
  be a set of affine linear functionals of the form $as+b$ with $a,b\in\mathbb{N}\cup\{0\}$ which admits duplications.
  Then we have
  \begin{align}
    f_{p,w}(s)&=\prod_{as+b\in S(\Delta;s;p,w)}\frac{1}{as+b-1},
\\
    g_{p,w}(s)&=\prod_{as+b\in S(\Phi_-;s;p,w)}\frac{1}{\xi(as+b)}.
  \end{align}
  
  Using the formula \eqref{eq:f1} and Lemma \ref{lm:rho},
  we have
  \begin{equation}
    \begin{split}
      S(A;-c_p-s;p,w)
      &=
      \{\langle \lambda_p,\alpha^\vee\rangle (-c_p-s)+\Ht\alpha^\vee~|~\alpha\in(w^{-1}A)\setminus\Delta_p\}
      \\
      &=
      \{\langle \lambda_p,-w_p\alpha^\vee\rangle s+\langle c_p\lambda_p-w_p\rho,-w_p\alpha^\vee\rangle~|~\alpha\in(w^{-1}A)\setminus\Delta_p\}
      \\
      &=\{\langle \lambda_p,\beta^\vee\rangle s+\langle \rho,\beta^\vee\rangle~|~\beta\in(w_pw^{-1}w_0A)\setminus\Delta_p\}
      \\
      &=S(A;s;p,w_0ww_p),
    \end{split}
  \end{equation}
which implies \eqref{eq:sym_fg1}.
  
  For \eqref{eq:sym_fg2},
  using $\varpi\Delta_p=\Delta_q$ and $\varpi A=A$, we have
  \begin{equation}
    \begin{split}
      S(A;s;p,\varpi^{-1}w\varpi)
      &=
      \{\langle \lambda_p,\alpha^\vee\rangle s+\Ht\alpha^\vee~|~\alpha\in(\varpi^{-1}w^{-1}\varpi A)\setminus\Delta_p\}
      \\
      &=
      \{\langle \varpi\lambda_p,\varpi\alpha^\vee\rangle s+\Ht\varpi\alpha^\vee~|~\alpha\in(\varpi^{-1}w^{-1}\varpi A)\setminus\Delta_p\}
      \\
      &=
      \{\langle \lambda_q,\beta^\vee\rangle s+\Ht\beta^\vee~|~\beta\in(w^{-1}A)\setminus\Delta_q\}
      \\
      &=S(A;s;q,w).
    \end{split}
  \end{equation}
\end{proof}
\begin{proof}[Proof of Proposition \ref{prop:main}]
  For $w\in W$,
  we have
  by Lemma \ref{lm:rho},
  \begin{equation}
    \begin{split}
      \langle w(\rho+(-c_p-s)\lambda_p)-\rho,\varpi_0T\rangle
      &=
      \langle w(\rho-c_p\lambda_p)-ws\lambda_p-\rho,-w_0T\rangle
      \\
      &=
      \langle -ww_p\rho-ww_ps\lambda_p-\rho,-w_0T\rangle
      \\
      &=
      \langle w_0ww_p(\rho+s\lambda_p)-\rho,T\rangle.
    \end{split}
  \end{equation}
  Hence
  using Proposition \ref{prop:explicit} and Lemma \ref{lm:sym_fg}, we obtain
  \begin{multline}
      Z_p(-c_p-s;\varpi_0T)
      \\
      \begin{aligned}
      &=
      \sum_{\substack{w\in W\\\Delta_p\subset w^{-1}(\Delta\cup \Phi_-)}}
      e^{\langle w(\rho+(-c_p-s)\lambda_p)-\rho,\varpi_0T\rangle}
      f_{p,w}(-c_p-s)g_{p,w}(-c_p-s)
      \\
      &=
      \sum_{\substack{w\in W\\\Delta_p\subset w^{-1}(\Delta\cup \Phi_-)}}
      e^{\langle w_0ww_p(\rho+s\lambda_p)-\rho,T\rangle}
      f_{p,w_0ww_p}(s)g_{p,w_0ww_p}(s)
      \\
      &=
      \sum_{\substack{v\in W\\\Delta_p\subset w_pv^{-1}w_0(\Delta\cup \Phi_-)}}
      e^{\langle v(\rho+s\lambda_p)-\rho,T\rangle}
      f_{p,v}(s)g_{p,v}(s),
    \end{aligned}
  \end{multline}
  which implies the first equality of \eqref{eq:main}
  by Lemma \ref{lm:bij_w} (1).
  
  As for the second equality,
  we have by Lemmas \ref{lm:sym_fg} and \ref{lm:bij_w} (2),
  \begin{equation}
    \begin{split}
      Z_q(s;\varpi T)
      &=
      \sum_{\substack{w\in W\\\Delta_q\subset w^{-1}(\Delta\cup \Phi_-)}}
      e^{\langle \varpi^{-1} w(\rho+s\lambda_q)-\rho,T\rangle}
      f_{q,w}(s)g_{q,w}(s)
      \\
      &=
      \sum_{\substack{w\in W\\\Delta_q\subset w^{-1}(\Delta\cup \Phi_-)}}
      e^{\langle \varpi^{-1}w\varpi(\rho+s\lambda_p)-\rho,T\rangle}
      f_{p,\varpi^{-1}w\varpi}(s)g_{p,\varpi^{-1}w\varpi}(s)
      \\
      &=
      \sum_{\substack{v\in W\\\Delta_q\subset \varpi v^{-1}\varpi^{-1}(\Delta\cup \Phi_-)}}
      e^{\langle v(\rho+s\lambda_p)-\rho,T\rangle}
      f_{p,v}(s)g_{p,v}(s).
    \end{split}
  \end{equation}
\end{proof}

From Proposition \ref{prop:explicit}, we see that $Z_p(s;T)$ has no $\xi$
functions in the denominator of each term.  In fact,
it is too much multiplied;
 $Z_p(s;T)$ can be
factorized by some $\xi$ functions and should be divided by them
 in order to obtain Weng's zeta function $\xi^{G/P}_{\mathbb{Q};o}(s;T)$.

Let
\begin{equation}
  H_{p,w}(s)=
  \biggl(
  \prod_{\alpha\in\Phi_w\setminus\Delta_p}
  \xi(\langle \lambda_p,\alpha^\vee\rangle s+\Ht\alpha^\vee)
  \biggr)
  \biggl(
  \prod_{\alpha\in \Phi_w}
  \frac{1}
  {\xi(\langle \lambda_p,\alpha^\vee\rangle s+\Ht\alpha^\vee+1)}
  \biggr),
\end{equation}
which is the term corresponding to $w$ in
$\omega^{G/P}_{\mathbb{Q}}(s;T)$ (see Proposition \ref{prop:exp_omega}).
Since
\begin{equation}
  \begin{split}
  \prod_{\alpha\in\Phi_w\setminus\Delta_p}
  \xi(\langle \lambda_p,\alpha^\vee\rangle s+\Ht\alpha^\vee)
  &=
  \xi(s+1)^{N_{p,w}(1,1)}
  \prod_{\alpha\in\Phi_w\setminus\Delta}
  \xi(\langle \lambda_p,\alpha^\vee\rangle s+\Ht\alpha^\vee)
  \\
  &=
  \xi(s+1)^{N_{p,w}(1,1)}
  \prod_{k=0}^\infty
  \prod_{h=2}^\infty
  \xi(ks+h)^{N_{p,w}(k,h)},
\end{split}
\end{equation}
and
\begin{equation}
  \begin{split}
  \prod_{\alpha\in\Phi_w}
  \xi(\langle \lambda_p,\alpha^\vee\rangle s+\Ht\alpha^\vee+1)
&=
  \prod_{k=0}^\infty
  \prod_{h=1}^\infty
  \xi(ks+h+1)^{N_{p,w}(k,h)}
\\
&=
  \prod_{k=0}^\infty
  \prod_{h=2}^\infty
  \xi(ks+h)^{N_{p,w}(k,h-1)},
\end{split}
\end{equation}
we have the expression
\begin{equation}
  \label{eq:mod_H_p}
  H_{p,w}(s)=
  \xi(s+1)^{N_{p,w}(1,1)}
  \prod_{k=0}^\infty
  \prod_{h=2}^\infty
  \xi(ks+h)^{N_{p,w}(k,h)-N_{p,w}(k,h-1)}.
\end{equation}

From this expression, we see that if $N_{p,w}(k,h)-N_{p,w}(k,h-1)<0$,
then $\xi(ks+h)^{N_{p,w}(k,h-1)-N_{p,w}(k,h)}$ appears in the
denominator of the term $H_{p,w}(s)$.  Let $\delta(a)=a$ if $a>0$, and
$\delta(a)=0$ otherwise.
In order to describe the minimal $\xi$ factor
 that cancels all
the denominators of $\omega^{G/P}_{\mathbb{Q}}(s;T)$,
we introduce
\begin{equation}
  \Tilde{M}_p(k,h)=\max_{\substack{w\in W\\ \Delta_p\subset w^{-1}(\Delta\cup \Phi_-)}}
\delta(N_{p,w}(k,h-1)-N_{p,w}(k,h))
\end{equation}
for $(k,h)\in\mathbb{Z}^2$ 
and
we define $D_p(s)$ by
\begin{equation}
\label{eq:def_D_p}
  D_p(s)
  =
  \prod_{k=0}^\infty
  \prod_{h=2}^\infty
  \xi(ks+h)^{N_p(k,h-1)-\Tilde{M}_p(k,h)}.
\end{equation}
By use of the definition \eqref{eq:def_F_p},
we see
\begin{equation}
  \label{eq:mod_F_p}
  \begin{split}
    F_p(s)
    &=
    \prod_{\alpha\in\Phi_+}
    \xi(\langle \lambda_p,\alpha^\vee\rangle s+\Ht\alpha^\vee+1)
    \\
    &=
    \prod_{k=0}^\infty
    \prod_{h=1}^\infty
    \xi(ks+h+1)^{N_p(k,h)}
    \\
    &=
    \prod_{k=0}^\infty
    \prod_{h=2}^\infty
    \xi(ks+h)^{N_p(k,h-1)}
  \end{split}
\end{equation}
and
\begin{equation}
\label{eq:DF}
  D_p(s)
  =
  F_p(s)
  \prod_{k=0}^\infty
  \prod_{h=2}^\infty
  \xi(ks+h)^{-\Tilde{M}_p(k,h)},
\end{equation}
so that $F_p(s)/D_p(s)$ is the minimal $\xi$ factor.
Note that \eqref{eq:def_M_p} is rewritten as
\begin{equation}
  M_p(k,h)=\max_{\substack{w\in W\\ \Delta_p\subset w^{-1}(\Delta\cup \Phi_-)}}(N_{p,w}(k,h-1)-N_{p,w}(k,h)).
\end{equation}

\begin{lemma} 
\label{lm:exp_M_p}
  \begin{enumerate}
  \item For $(k,h)\in\mathbb{Z}^2$ with $h\geq1$,
\begin{equation}
  M_p(k,h)=\Tilde{M}_p(k,h).
\end{equation}
\item For $(k,h)\in\mathbb{Z}^2$,
\begin{equation}
  N_p(k,kc_p-h)-M_p(k,kc_p-h+1)
=N_p(k,h-1)-M_p(k,h).
\end{equation}
\item For $(k,h)\in\mathbb{Z}^2$,
  \begin{equation}
M_p(k,h)=M_q(k,h),
\end{equation}
where $q\in\Aut(\Gamma)p$.
\end{enumerate}
\end{lemma}
\begin{proof}
  \begin{enumerate}
  \item Note that $\Delta_p\subset(\Delta\cup\Phi_-)$.
    For $l\geq0$, we have
    \begin{equation}
      N_{p,\Id}(k,l)=
      \sharp\{\alpha\in \Phi_-~|~\langle\lambda_p,\alpha^\vee\rangle=k,\Ht\alpha^\vee=l\}=0,
    \end{equation}
    which implies
    \begin{equation}
      N_{p,\Id}(k,h-1)-N_{p,\Id}(k,h)=0.
    \end{equation}
    Hence $\delta(x)$ can be replaced by $x$ in $\Tilde{M}_p(k,h)$.
  \item From Lemma \ref{lm:Nw} (1), we have
    \begin{multline}
      M_p(k,kc_p-h+1)
      \\
      \begin{aligned}
        &=\max_{\substack{w\in W\\ \Delta_p\subset w^{-1}(\Delta\cup \Phi_-)}}
        (N_{p,w}(k,kc_p-h)-N_{p,w}(k,kc_p-h+1))
        \\
        &=\max_{\substack{w\in W\\ \Delta_p\subset w^{-1}(\Delta\cup \Phi_-)}}
        (N_p(k,h)-N_p(k,h-1)-N_{p,w_0ww_p}(k,h)+N_{p,w_0ww_p}(k,h-1)).
      \end{aligned}
    \end{multline}
    Hence
    by Lemmas \ref{lm:bij_w} (1) and \ref{lm:N} (1), 
    \begin{multline}
      N_p(k,kc_p-h)-M_p(k,kc_p-h+1)
      \\
      \begin{aligned}
        &=N_p(k,h)-  
        \max_{\substack{w\in W\\ \Delta_p\subset w^{-1}(\Delta\cup \Phi_-)}}
        (N_p(k,h)-N_p(k,h-1)-N_{p,w_0ww_p}(k,h)+N_{p,w_0ww_p}(k,h-1))
        \\
        &=N_p(k,h-1)-\max_{\substack{w\in W\\ \Delta_p\subset w^{-1}(\Delta\cup \Phi_-)}}(N_{p,w_0ww_p}(k,h-1)-N_{p,w_0ww_p}(k,h))
        \\
        &=N_p(k,h-1)-M_p(k,h).
      \end{aligned}
    \end{multline}
  \item
    By 
    Lemmas
    \ref{lm:bij_w} (2) and \ref{lm:Nw} (2), we have the result.
  \end{enumerate}
\end{proof}

\begin{proof}[Proof of Theorem \ref{thm:exp_xi}]
Since by 
\eqref{eq:DF},
 $F_p(s)/D_p(s)$ is the minimal $\xi$ factor for $\omega^{G/P}_{\mathbb{Q}}(s;T)$,
we have \eqref{eq:exp_xi} by Lemma \ref{lm:exp_M_p} (1) and
the expression
\begin{equation}
  \label{eq:exp_xi2}
  \xi^{G/P}_{\mathbb{Q};o}(s;T)=\Bigl(\frac{F_p(s)}{D_p(s)}\Bigr)\omega^{G/P}_{\mathbb{Q}}(s;T).
\end{equation}
\end{proof}

\begin{lemma}
\label{lm:DD}
\begin{equation}
\label{eq:DD}
  \begin{split}
    D_p(-c_p-s)&=D_p(s)\\
    &=D_q(s),
  \end{split}
\end{equation}
where $q\in\Aut(\Gamma)p$.
\end{lemma}
\begin{proof}
  We show the first equality.
  We use 
  \begin{align}
    D^{(0)}&=
    \prod_{h=2}^\infty
    \xi(h)^{N_p(0,h-1)-M_p(0,h)},
    \\
    D^{(1)}(s)&=
    \prod_{k=1}^\infty
    \prod_{h=2}^\infty
    \xi(ks+h)^{N_p(k,h-1)-M_p(k,h)}\\
    \notag
    &=
    \prod_{k=1}^\infty
    \prod_{h=-\infty}^\infty
    \xi(ks+h)^{N_p(k,h-1)-M_p(k,h)},
    \notag
  \end{align}
  since $N_{p,w}(k,h-1)=0$ and $M_p(k,h)=0$ for $k\geq 1$ and $h\leq
  1$.  Note that $D_p(s)=D^{(0)}D^{(1)}(s)$.  It is sufficient to show
  $D^{(1)}(-c_p-s)=D^{(1)}(s)$.  We have
  \begin{equation}
    \begin{split}
      D^{(1)}(-c_p-s)
      &=
      \prod_{k=1}^\infty
      \prod_{h=-\infty}^\infty
      \xi(-kc_p-ks+h)^{N_p(k,h-1)-M_p(k,h)}
      \\
      &=
      \prod_{k=1}^\infty
      \prod_{h=-\infty}^\infty
      \xi(ks+kc_p-h+1)^{N_p(k,h-1)-M_p(k,h)}
      \\
      &=
      \prod_{k=1}^\infty
      \prod_{h=-\infty}^\infty
      \xi(ks+h)^{N_p(k,kc_p-h)-M_p(k,kc_p-h+1)}.
    \end{split}
  \end{equation}
  By Lemma \ref{lm:exp_M_p} (2),
 we obtain
  \begin{equation}
    D^{(1)}(-c_p-s)
    =
    \prod_{k=1}^\infty
    \prod_{h=-\infty}^\infty
    \xi(ks+h)^{N_p(k,h-1)-M_p(k,h)}
    =
    D^{(1)}(s)
  \end{equation}
  and hence the result.

  The second equality of \eqref{eq:DD} follows from
  the definition \eqref{eq:def_D_p}
  and Lemmas 
  \ref{lm:N} (2) and \ref{lm:exp_M_p} (3).
\end{proof}

\begin{proof}[Proof of Theorem \ref{thm:main1}]
By \eqref{eq:exp_xi2} and \eqref{eq:def_Z_p},
we rewrite
  \begin{equation}
    \xi^{G/P}_{\mathbb{Q};o}(s;T)=\Bigl(\frac{F_p(s)}{D_p(s)}\Bigr)\omega^{G/P}_{\mathbb{Q}}(s;T)=\frac{Z_p(s;T)}{D_p(s)}.
  \end{equation}
Then the functional equation
follows 
from Proposition \ref{prop:main} and Lemma \ref{lm:DD}.
\end{proof}

\providecommand{\bysame}{\leavevmode\hbox to3em{\hrulefill}\thinspace}
\providecommand{\MR}{\relax\ifhmode\unskip\space\fi MR }
\providecommand{\MRhref}[2]{%
  \href{http://www.ams.org/mathscinet-getitem?mr=#1}{#2}
}
\providecommand{\href}[2]{#2}

\end{document}